\documentclass[reqno]{amsart}

\usepackage{amsmath}
\usepackage{amsfonts}
\usepackage{amssymb,enumerate}
\usepackage{amsthm}
\usepackage[all]{xy}
\usepackage{rotating}
\usepackage{hyperref}
\usepackage{color}
\usepackage{textcomp}

\theoremstyle{plain}
\newtheorem{lem}{Lemma}[section]

\newtheorem{prop}[lem]{Proposition}
\newtheorem{thm}[lem]{Theorem}

\theoremstyle{definition}
\newtheorem{defn}[lem]{Definition}

\newtheorem{rmk}[lem]{Remark}
\newtheorem{construction}[lem]{Construction}

\newcommand{\depth}{\operatorname{depth}}	
	
\newcommand{\amp}{\operatorname{amp}}

\newcommand{\lotimes}{\otimes^{\mathbf{L}}} 
\newcommand{\biglotimes}{\bigotimes^{\mathbf{L}}}

\newcommand{\tor}{\operatorname{Tor}}
\newcommand{\im}{\operatorname{Im}}
\newcommand{\shift}{\mathsf{\Sigma}}

\newcommand{\Ker}{\operatorname{Ker}}

\newcommand{\ideal}[1]{\mathfrak{#1}}

\newcommand{\fm}{\ideal{m}}

\newcommand{\wt}{\widetilde}

\newcommand{\bbz}{\mathbb{Z}}

\newcommand{\xra}{\xrightarrow}

\renewcommand{\geq}{\geqslant}
\renewcommand{\leq}{\leqslant}
\renewcommand{\ker}{\Ker}

\numberwithin{equation}{lem}

\begin{document}

\bibliographystyle{amsplain}

\author{Sean K. Sather-Wagstaff}

\address{School of Mathematical and Statistical Sciences,
Clemson University,
O-110 Martin Hall, Box 340975, Clemson, S.C. 29634
USA}

\email{ssather@clemson.edu}

\urladdr{https://ssather.people.clemson.edu/}

\author{Hannah Altmann}

\address{College of Arts and Sciences, 
Dakota State University,
146L Habeger Science Center, 
820 N Washington Ave,
Madison, SD 57042
USA}

\email{hannah.altmann@dsu.edu}

\title{On Gerko's Strongly Tor-independent Modules}

\begin{abstract} 
Gerko proves that if an artinian local ring $(R,\mathfrak{m}_R)$ possesses a sequence of strongly Tor-independent modules of length $n$, then $\mathfrak{m}_R^n\neq 0$. This generalizes readily to Cohen-Macaulay rings. We present a version of this result for non-Cohen-Macaulay rings.
\end{abstract}

\maketitle

\section{Introduction}\label{sec201115a}

We are interested in how existence of certain sequences of modules over a local ring $(R,\fm_R)$ imposes restrictions on $R$. Specifically, we investigate what Gerko~\cite{gerko:sdc} calls strongly Tor-independent $R$-modules: A sequence $N_1,\ldots,N_n$ of $R$-modules is \emph{strongly Tor-independent} provided $\tor_{\geq 1}^R(N_{j_1}\otimes_R \cdots\otimes_R N_{j_t},N_{j_{t+1}})=0$ for all distinct $j_1,\ldots,j_{t+1}$. Gerko is led to this notion in his study of Foxby's semidualizing modules~\cite{foxby:gmarm} and Christensen's semidualizing complexes~\cite{christensen:scatac}. In particular, Gerko~\cite[Theorem 4.5]{gerko:sdc} proves that if $R$ is artinian and possesses a sequence of strongly Tor-independent modules of length $n$, then $\fm_R^n\neq 0$. This generalizes readily from artinian rings to Cohen-Macaulay rings. 

Our goal in this paper is to prove the following non-Cohen-Macaulay version of Gerko's result.  

\begin{thm}\label{thm200220c} Assume $(R,\fm_R)$ is a local ring. If $N_1,\ldots, N_n$ are strongly Tor-independent $R$-modules, then $n\leq \operatorname{ecodepth}(R)$.
\end{thm}

Here $\operatorname{ecodepth}(R)=\beta^R_0(\fm_R)-\depth(R)$ is the \emph{embedding codepth of $R$}, where $\beta^R_0(\fm_R)$ is the minimal number of generators of $\fm_R$. Note that our result does not recover Gerko's, but compliments it. Our proof is the subject of Section~\ref{sec201129a} below.

Part of the proof of our result is modeled on Gerko's proof with one crucial difference: where Gerko works over an artinian ring, we work over a finite dimensional DG algebra. See Sections~\ref{sec201115c} and \ref{sec201115b} for background material and foundational results, including our DG version of Gerko's notion of strong Tor-independence. Theorem~\ref{thm200220a} is our main result in the DG context, which is the culmination of Section~\ref{sec201122a}. Our proof relies on a DG syzygy construction of Avramov et al. \cite{avramov:htecdga}.

\section{DG Homological Algebra}\label{sec201115c}

Let $R$ be a nonzero commutative noetherian ring with identity. We work with $R$-complexes indexed homologically, i.e., where an $R$-complex $X$ has differential $\partial_i^X\colon X_i\to X_{i-1}$. The \emph{supremum} and \emph{infimum} of $X$ are respectively \begin{align*} \sup(X)&=\sup\{i\in \mathbb{Z}\mid X_i\neq 0\} & \inf(X)&=\inf\{i\in \mathbb{Z}\mid X_i\neq 0\}.\end{align*} The \emph{amplitude} of $X$ is $\amp(X)=\sup(X)-\inf(X).$ Frequently we consider these invariants applied to the total homology $H(X)$, e.g., as $\sup(H(X))$. 

As we noted in the introduction, the proof of Theorem~\ref{thm200220c} uses DG techniques which we summarize next. See, e.g., \cite{avramov:dgha,felix:rht} for more details.

A \emph{differential graded (DG) $R$-algebra} is an $R$-complex $A$ equipped with an $R$-linear chain map $A\otimes_R A\to A$ denoted $a\otimes a'\mapsto aa'$ that is unital, associative, and graded commutative. We simply write \emph{DG algebra} when $R=\mathbb{Z}$. The chain map condition here implies that this multiplication is also distributive and satisfies the \emph{Leibniz Rule}: $\partial(aa')=\partial(a)a'+(-1)^{|a|}a\partial(a')$ where $|a|$ is the homological degree of $a$. We say that $A$ is \emph{positively graded} provided $A_i=0$ for all $i<0$. For example, the trivial $R$-complex $R$ is a positively graded DG $R$-algebra, so too is every Koszul complex over $R$, using the wedge product. The \emph{underlying algebra} associated to $A$ is the $R$-algebra $A^{\natural}=\bigoplus_{i\in \mathbb{Z}} A_i$.

A positively graded DG $R$-algebra $A$ is \emph{local} provided $H_0(A)$ is noetherian, each $H_0(A)$-module $H_i(A)$ is finitely generated for all $i\geq 0$, $R$ is local, and the ring $H_0(A)$ is a local $R$-algebra.

Let $A$ be a DG $R$-algebra. A \emph{DG $A$-module} is an $R$-complex $X$ equipped with an $R$-linear chain map $A\otimes_R X\to X$ denoted $a\otimes x\mapsto ax$ that is unital and associative. For instance DG $R$-modules are precisely $R$-complexes. We say that $X$ is \emph{homologically bounded} if $\amp(H(X))<\infty$, and we say that $X$ is \emph{homologically finite} if $H(X)$ is finitely generated over $H_0(A)$. We write $\shift^n X$ for the \emph{$n$th shift of $X$} obtained by $(\shift^n X)_i=X_{i-n}$. Quasiisomorphisms between $R$-complexes, i.e., chain maps that induce isomorphisms on the level of homology, are identified with the symbol $\simeq$. 

Let $A$ be positively graded and let $X$ be a DG $A$-module such that $\inf(X)>-\infty$. We say that $X$ is \emph{semifree} if the underlying $A^{\natural}$-module $X^{\natural}$ is free. In this case a \emph{semibasis} for $X$ is a set of homogeneous elements of $X$ that is a basis for $X^{\natural}$ over $A^{\natural}$. A \emph{semifree resolution} of a DG $A$-module $Y$ with $\inf(H(Y))>-\infty$ is a quasiisomorphism $F\xra \simeq Y$ such that $F$ is semifree. The derived tensor product of DG $A$-modules $Y$ and $Z$ is $Y\lotimes_A Z\simeq F\otimes_A Z$ where $F\xra\simeq Y$ is a semifree resolution of $Y$. We say that $Y$ is \emph{perfect} if it has a semifree resolution $F\xra \simeq Y$ such that $F$ has a finite semibasis. 

Let $A$ be a local DG $R$-algebra, and let $Y$ be a homogically finite DG $A$-module. By~\cite[Proposition B.7]{avramov:htecdga} $Y$ has a \emph{minimal semifree resolution}, i.e., a semifree resolution $F\xra \simeq Y$ such that the semibasis for $F$ is finite in each homological degree and $\partial^F(F)\subseteq \fm_A F$.

\section{Perfect DG Modules and Tensor Products}\label{sec201115b}

Throughout this section, let $A$ be a positively graded commutative homologically bounded DG algebra, say $\amp(H(A))=s$,
and assume that $A\not\simeq 0$. 

\

Most of this section focuses on four foundational results on perfect DG modules.

\begin{lem}\label{lem180324b}
Let  $L$ be a non-zero semifree DG $A$-module with a  semibasis $B$ concentrated in a single degree $n$.
Then $L\cong\shift^nA^{(B)}$.
In particular, $\inf(H(L))=n$ and $\sup(H(L))=s+n$ and $\amp(H(L))=s$.
\end{lem}

\begin{proof}
It suffices to prove that $L\cong\shift^nA^{(B)}$. Apply an appropriate shift to assume without loss of generality that $n=0$.

The semifree/semibasis assumptions tell us that every
element $x\in L$ has the form $\sum_{e\in B}^{\text{finite}}a_ee$; the linear independence of the semibasis
tells us that this representation is essentially unique.
Since $A$ is positively graded, we have $L_{-1}=0$, so
$\partial^L(e)=0$ for all $e\in B$. Hence, the Leibniz rule for $L$ implies that
$$
\partial^L\left(\sum_{e\in B}^{\text{finite}}a_ee\right)=\sum_i^{\text{finite}}\partial^A(a_e)e+\sum_i^{\text{finite}}(-1)^{|a_e|}a_e\partial^L(e)
=\sum_i^{\text{finite}}\partial^A(a_e)e.$$
From this, it follows that the map $A^{(B)}\to L$ given by the identity on $B$ is an isomorphism.
\end{proof}

\begin{prop}\label{prop180324a}
Let $L$ be a non-zero semifree DG $A$-module with a  semibasis $B$ concentrated in degrees $n,n+1,\ldots,n+m$ where $n,m\in\bbz$
and $m
\geq 0$.
Then $\inf(H(L))\geq n$ and $\sup(H(L))\leq s+n+m$,
so $\amp(H(L))\leq s+m$.
\end{prop}

\begin{proof}
It suffices to show that $\inf(H(L))\geq n$ and $\sup(H(L))\leq s+n+m$. We induct on $m$. 
The base case $m=0$ follows from Lemma~\ref{lem180324b}.

For the induction step, assume that $m\geq 1$ and that the result holds for 
semifree DG $A$-modules with  semibasis concentrated in degrees $n,n+1,\ldots,n+m-1$.
Set
$$B'=\{e\in B\mid |e|<n+m\}$$
and let $L'$ denote the semifree submodule of $L$ spanned over $A$ by $B'$. 
(See the first paragraph of the proof of~\cite[Proposition~4.2]{avramov:htecdga} for further details.)
Note that $L'$ has semibasis $B'$ concentrated in degrees $n,n+1,\ldots,n+m-1$.
In particular, our induction assumption applies to $L'$ to give
$\inf(H(L'))\geq n$ and $\sup(H(L'))\leq s+n+m-1$.

If $L=L'$, then we are done by our induction assumption. So assume that $L\neq L'$. 
Then the quotient $L/L'$ is semifree and non-zero with semibasis concentrated in degree $n+m$. 
So, Lemma~\ref{lem180324b} implies that 
$\inf(H(L/L'))=n+m$ and $\sup(H(L/L'))= s+n+m$.
Now, consider the short exact sequence
\begin{equation}\label{eq201115d}
0\to L'\to L\to L/L'\to 0.
\end{equation}
The desired conclusions for $L$ follow from the associated long exact sequence in homology. 
\end{proof}

Now, we use the preceding two results to analyze derived tensor products.

\begin{lem}\label{lem180521b} Let $L$ be a non-zero semifree DG $A$-module with a semibasis $B$ concentrated in a single degree, say $n$, and let $Y$ be a homologically bounded DG $A$-module. Then $L\lotimes_A Y\simeq \shift^nY^{(B)}$. In particular, $\inf(H(L\lotimes_A Y)) =\inf(H(Y))+n$ and $\sup(H(L\lotimes_A Y)) =\sup(H(Y))+n$ and $\amp(H(L\lotimes_A Y)) =\amp(H(Y))$.
\end{lem}

\begin{proof} Immediate from Lemma~\ref{lem180324b}.
\end{proof}

\begin{prop}\label{prop180521c} Let $L$ be a non-zero semifree DG $A$-module with a semibasis $B$ concentrated in degrees $n$, $n+1$,$\ldots n+m$ where $n,m\in \bbz$ and $m\geq 0$, and let $Y$ be a homologically bounded DG $A$-module. Then $\inf(H(L\lotimes_A Y)) \geq \inf(H(Y))+n$ and $\sup(H(L\lotimes_A Y)) \leq \sup(H(Y))+n+m$, so $\amp(H(L\lotimes_A Y)) \leq \amp(H(Y))+m$.
\end{prop}

\begin{proof} As in the proof of Proposition~\ref{prop180324a}, we induct on $m$. The base case $m=0$ follows from Lemma~\ref{lem180521b}.

For the induction step, assume $m\geq 1$ and the result holds for semifree DG $A$-modules with semibasis concentrated in degrees $n$, $n+1$, $\ldots, n+m-1$ and $Y\in D_b(A)$. We work with the notation from the proof of Proposition~\ref{prop180324a}, and we assume that $L\neq L'$. The exact sequence \eqref{eq201115d} of semi-free DG modules gives rise to the following distinguished triangle in $\mathsf{D}(A)$. 
$$L'\lotimes_A Y\to L\lotimes_A Y\to (L/L')\lotimes_A Y\to$$ Another long exact sequence argument gives the desired conclusion.
\end{proof}
 
We close this section with our DG version of strongly Tor-independent modules.

\begin{defn} The DG $A$-modules $K_1,\ldots, K_n$ are said to be \emph{strongly Tor-independent} if for any subset $I\subset \{1,\ldots, n\}$ we have $\amp(H(\lotimes_{i\in I} K_i))\leq s$.
\end{defn}

\begin{rmk}\label{rmk180531a}  It is worth noting that the definition of $K_1,\ldots, K_n$ being strongly Tor-independent includes $\amp(H(K_i)))\leq s$ for all $i=1,\ldots, n$. Also, if $K_1,\ldots, K_n$ are strongly Tor-independent, then so is any reordering by the commutativity of tensor products.
\end{rmk}

\section{Syzygies and Strongly Tor-independent DG modules}\label{sec201122a}

Throughout this section, let $(A,\fm_A)$ be a local homologically bounded DG algebra, say $\amp(H(A))=s$,
and assume that $A\not\simeq 0$ and $\fm_A=A_+$. It follows that $A_0$ is a field.

\

The purpose of this section is to provide a DG version of part of a result of Gerko~\cite[Theorem 4.5]{gerko:sdc}. 
Key to this is the following slight modification of the syzygy construction of Avramov et al.\ mentioned in the introduction.

\begin{construction}\label{const190804a}  Let $K$ be a homologically finite DG $A$-module. Let $F\simeq K$ be a minimal semifree resolution of $K$, and let $E$ be a semibasis for $F$. Let $F^{(p)}$ be the semifree DG $A$-submodule of $F$ spanned by $E_{\leqslant p}:=\cup_{m\leqslant p} E_m$. 

Set $t=\sup(H(K))$, and consider the soft truncation $\wt{K}=\tau_{\leqslant r}(F)$ for a fixed integer $r\geq t$. Note that the natural morphism $F\to \wt{K}$ is a surjective quasiisomorphism of DG $A$-modules, so we have $\wt{K}\simeq F\simeq K$. Next, set $L=F^{(r)}$, which is semifree with a finite semibasis $E_{\leqslant r}$. Furthermore, the composition $\pi$ of the natural morphisms $L=F^{(r)}\to F\to \wt{K}$ is surjective because the morphism $F\to \wt{K}$ is surjective, the morphism $L\to F$ is surjective in degrees $\leqslant r$, and $\wt{K}_i=0$ for all $i>r$. Set $\operatorname{Syz}_r(K)=\ker(\pi)\subseteq L$ and let $\alpha\colon\operatorname{Syz}_r(K)\to L$ be the inclusion map.
\end{construction}

\begin{prop}\label{prop190804b} Let $K$ be a homologically finite DG $A$-module. With the notation of Construction~\ref{const190804a}, there is a short exact sequence of morphisms of DG $A$-modules
\begin{equation}\label{eq190812a}
0 \to \operatorname{Syz}_r(K) \xrightarrow{\alpha} L \xrightarrow{\pi} \wt{K}\to 0\end{equation}
such that $L$ is semifree with a finite semibasis and where $\wt{K}\simeq K$ and $\im(\alpha)\subseteq A_+L$.
\end{prop}

\begin{proof} Argue as in the proof of \cite[Proposition~4.2]{avramov:htecdga}. 
\end{proof}

Our proof of Theorem~\ref{thm200220c} hinges on the behavior for syzygies documented in the following four results.

\begin{lem}\label{prop180523b} Let $K$ be a homologically finite DG $A$-module with $\amp(H(K))\leq s$ and $K'=\operatorname{Syz}_r(K)$ where $r\geq \sup(H(K))$. Then $\sup(H(K'))\leq s+r$ and $\inf(H(K'))\geq r$. Therefore, $\amp(H(K'))\leq s$.
\end{lem}

\begin{proof} Use the notation from Construction~\ref{const190804a}. Then $\sup(H(L))\leq s+r$ by Proposition~\ref{prop180324a}. Also, by definition we have $\sup(H(\wt{K}))=\sup(H(K))\leq r\leq r+s$. The long exact sequence in homology coming from~\eqref{eq190812a} implies $\sup(H(K'))\leq s+r$. Also, $\inf(H(K'))\geq \inf(K')\geq r$ because $\pi_i$ is an isomorphism for all $i<r$ by Construction~\ref{const190804a}. So, $\amp(H(K'))=\sup(H(K'))-\inf(H(K'))\leq s+r-r=s$. 
\end{proof}

\begin{prop}\label{prop180523a} Let $K$ be a homologically finite DG $A$-module and set $K'=\operatorname{Syz}_r(K)$ where $r\geq \sup(H(K))$. Let $Y$ be a homologically bounded DG $A$-module and assume that $K,Y$ are strongly Tor-independent. Then $\sup(H(K'\lotimes_AY))\leq \sup(H(Y))+r$ and $\inf(H(K'\lotimes_AY))\geq \inf(H(Y))+r$. So, $\amp(H(K'\lotimes_A Y))\leq s$; in particular, $K',Y$ are strongly Tor-independent.
\end{prop}

\begin{proof} Let $G\xra\simeq Y$ be a semifree resolution of $Y$. Let $\wt{K}$ and $L$ be as in Construction~\ref{const190804a}. Since $K,Y$ are strongly Tor-independent we have $\sup(H(\wt{K}\otimes_A G))\leq s$. Also, Proposition~\ref{prop180521c} implies $\sup(H(L\otimes_A G))\leq \sup(H(Y))+r$. To conclude the proof, consider the short exact sequence

\begin{equation}\label{eq180524a}
0\to K'\otimes_A G\to L\otimes_A G\to \wt{K}\otimes_A G\to 0
\end{equation}

\noindent and argue as in the proof of Lemma~\ref{prop180523b}.
\end{proof}

\begin{prop}\label{prop180605b} Let $K_1, K_2,\ldots, K_n$ be strongly Tor-independent, homologically finite DG $A$-modules for $n\in \bbz^{+}$ and $K_i'=\operatorname{Syz}_{r_i}(K_i)$ where $r_i\geq \sup(H(K_i))$. Then $K_1',\ldots,K_m',K_{m+1},\ldots, K_n$ are strongly Tor-independent for all $m=1,\ldots,n$.
\end{prop}

\begin{proof} Induct on $m$ using Proposition~\ref{prop180523a}.
\end{proof}

\begin{prop}\label{prop190917} Let $K_1,K_2,\ldots,K_{j}$ be strongly Tor-independent DG $A$-modules, and set $K_i'=\operatorname{Syz}_{r_i}(K_i)$ where $r_i\geq \sup(H(K_i))$ for $i=1,2,\ldots, j$. If $\fm_A^n=0$, then $\fm_{H(A)}^{n-j} H(\lotimes_{i=1,\ldots,j} K_i')=0$.
\end{prop}

\begin{proof} Shift $K_i$ if necessary to assume without loss of generality that $\inf(H(K_i))=0$ for $i=1,\ldots,j$. For $i=1,\ldots, j$ let $G_i\xra\simeq K'_i$ be semifree resolutions, and consider the following diagram with notation as in Construction~\ref{const190804a}. \begin{equation}\label{eq180803a}
\xymatrix@C=2cm{0 \ar[r]& K_i' \ar[r]^{\subseteq} & L_i \ar[r] & \wt{K_i}\ar[r] & 0\\ & G_i \ar[u]^{\simeq}  \ar[ur]^{\alpha_i} & & & }
\end{equation}

\noindent Notice, $\text{Im}(\alpha_i)\subseteq K_i'\subseteq \fm_A L_i$ for $i=1,2,\ldots, j$. 

Set $\mathcal{G}=\otimes_{i=1,\ldots,j-1}G_i$ and consider the following commutative diagram
$$\xymatrix@C=2cm{\left(\lotimes_{i=1,\ldots,j-1} K_i'\right)\lotimes_A K_{j}'\simeq \mathcal{G}\otimes_A G_{j} \ar[r]^-----{\beta} \ar[d]^{\mathcal{G}\otimes \alpha_{j}} \ar[rd]^>>>>>>>>>>>>>>>{\theta \otimes \alpha_{j}} & \fm_A^{j}((\otimes_{i=1,\ldots,j-1}L_i)\otimes_A L_{j}) \ar[d]^{\subseteq}\\ \mathcal{G}\otimes_A L_{j} \ar[r]_{\theta \otimes L_{j}} & (\otimes_{i=1,\ldots,j-1}L_i)\otimes_A L_{j}}$$
where $\theta=\otimes_{i=1,\ldots,j-1}\alpha_i$ and $\beta$ is induced by $\theta\otimes \alpha_j$.

Claim: $H(\beta)$ is 1-1. Notice that $H_i(\mathcal{G}\otimes_A G_j)=0$ for all $i<r_1+\ldots+r_j$, so it suffices to show that $H_i(\beta)$ is 1-1 for all $i\geq r_1+\ldots+r_{j}$. To this end it suffices to show $H_i(\mathcal{G} \otimes \alpha_{j})$ and $H_i(\theta\otimes L_{j})$ are 1-1 for all $i\geq r_1+\ldots+r_{j}$. First we show this for $H_i(\mathcal{G}\otimes \alpha_{j})$. Consider the short exact sequence
\begin{equation}\label{eq180814}
0 \to \mathcal{G}\otimes_A G_{j} \xrightarrow{\mathcal{G}\otimes \alpha_{j}} \mathcal{G}\otimes_A L_{j} \to \mathcal{G}\otimes_A\wt{K_{j}}\to 0. 
\end{equation}

\noindent Proposition~\ref{prop180523a} implies \begin{align*}\sup(H(\mathcal{G}\otimes_A \wt{K_{j}}))&\leq r_1+\ldots+r_{j-1}+\sup(H(\wt{K_{j}})) \leq r_1+\ldots+r_{j}.\end{align*} Thus, the long exact sequence in homology associated to \eqref{eq180814} implies $H_i(\mathcal{G}\otimes_A \alpha_{j})$ is 1-1 for all $i\geq r_1+\ldots+r_{j}$ as desired.

Next, we show $H_i(\theta\otimes L_{j})$ is 1-1 for $i\geq r_1+\ldots+r_{j}$. Consider the exact sequence
\begin{equation}\label{eq190812}
0 \to \mathcal{G}\otimes_A L_{j} \xrightarrow{\theta\otimes L_j} (\otimes_{i=1,\ldots,j-1}L_i)\otimes_A L_{j} \to  (\otimes_{i=1,\ldots,j-1}\wt{K}_i)\otimes_A L_{j}\to 0. 
\end{equation}
\noindent The first inequality in the next display follows from Proposition~\ref{prop180521c}  \begin{align*}\sup(H((\otimes_{i=1,\ldots,j-1}\wt{K}_i)\otimes_A L_{j}))&\leq \sup(H(\otimes_{i=1,\ldots,j-1}\wt{K}_i))+r_{j}\\&\leq r_1+\ldots +r_{j-1}+r_{j}.\end{align*} Thus, the long exact sequence in homology associated to \eqref{eq190812} implies $H_i(\theta\otimes L_{j})$ is 1-1 for all $i\geq r_1+\ldots+r_{j}$. This establishes the claim.

To complete the proof it remains to show $\fm^{n-j}_{H(A)} H((\lotimes_{i=1,\ldots,j-1} K_i')\lotimes_A K_{j}')=0.$ Since $H(\beta)$ is 1-1, we have $H((\lotimes_{i=1,\ldots,j-1} K_i')\lotimes_A K_{j}')$ isomorphic to a submodule of $H(\fm_A^{j}((\otimes_{i=1,\ldots,j-1}L_i)\otimes_A L_{j}))$. So it suffices to show that $\fm_{H(A)}^{n-j}$ annihilates $H(\fm_A^j((\otimes_{i=1,\ldots,j-1}L_i)\otimes_A L_{j}))$; this annihilation holds because $\fm_A^n=0$.
\end{proof}

Here is the aforementioned version of part of~\cite[Theorem 4.5]{gerko:sdc}. 

\begin{thm}\label{thm200220a} Let $K_1,\ldots, K_n$ be strongly Tor-independent non-perfect DG $A$-modules. Then $\fm_A^n\not=0$, therefore, $n\leq s$.
\end{thm}

\begin{proof} Suppose $\fm_A^n=0$. Proposition~\ref{prop190917} implies that $0=\fm_{H(A)}^{0} H(\lotimes_{i=1,\ldots,n} K_i')=H(\lotimes_{i=1,\ldots,n} K_i').$ Since each $K_i$ has a minimal resolution for $i=1,\ldots, n$, we must have $H(K_l')=0$ for some $l$. Hence, $K_l$ has a semifree basis concentrated in a finite number of degrees. This contradicts our assumption that $K_i$ is not perfect for $i=1,\ldots ,n$. Therefore, $\fm_A^n\not=0.$

Now we show $n\leq s$. Soft truncate $A$ to get $A'\simeq A$ such that $\sup(A')=s$. Thus, $\fm_{A'}^{s+1}=0$. The sequence of $n$ strongly Tor-independent non-perfect DG $A$-modules gives rise to a sequence of $n$ strongly Tor-independent non-perfect DG $A'$-modules. Since $\fm_{A'}^n\not = 0$ and $\fm_{A'}^{s+1}=0$, we have $n\leq s$.
\end{proof}

\section{Proof of Theorem~\ref{thm200220c}}\label{sec201129a}

Induct on $\depth(R)$. 

Base Case: $\depth(R)=0$. Let $K$ denote the Koszul complex over $R$ on a minimal generating sequence for $\fm_R$. The condition $\depth(R)=0$ implies \begin{equation}\label{eqn201124a}\amp(H(K))=\operatorname{ecodepth}(R)=\amp(K).\end{equation} 

Claim: The sequence $K\lotimes_R N_1, \ldots, K\lotimes_R N_n$ is a strongly Tor-independent sequence of DG $K$-modules.
To establish the claim we compute derived tensor products where both $\biglotimes$ are indexed by $i\in I$: \begin{equation*} \textstyle\biglotimes_{K} (K\lotimes_R N_i) \simeq K\lotimes_R (\biglotimes_{R} N_i).\end{equation*} From this we get the first equality in the next display. \begin{align*} \textstyle\amp(H(\biglotimes_K (K\lotimes_R N_i)))&=\textstyle \amp(H(K\lotimes_R (\biglotimes_R N_i)))\\&=\textstyle\amp(H(K\otimes_R (\bigotimes_{i\in I} N_i)))\\&\leq\textstyle\amp(K\otimes_R (\bigotimes_{i\in I} N_i))\\&=\amp(K)\\&\textstyle=\amp(H(K))\end{align*} The second equality comes from the strong Tor-independence of the original sequence. The inequality and the third equality are routine, and the final equality is by~\eqref{eqn201124a}. This establishes the claim. 

A construction of Avramov provides a local homologically bounded  DG algebra $(A,\fm_A)$ such that $A\simeq K\not\simeq 0$ and $\fm_A=A_+$; see \cite{MR1132435,nasseh:survey}. The strongly Tor-independent sequence $K\lotimes_R N_1, \ldots, K\lotimes_R N_n$ over $K$ gives rise to a strongly Tor-independent sequence $M_1,\ldots,M_n$ over $A$. Now, Theorem~\ref{thm200220a} and \eqref{eqn201124a} imply $n\leq \amp(H(A))=\amp(H(K))=\operatorname{ecodepth}(R)$. This concludes the proof of the Base Case.

Inductive Step: Assume $\depth(R)>0$ and the result holds for local rings $S$ with $\depth(S)=\depth(R)-1$. For $i=1,\ldots, n$ let $N_i'$ be the first syzygy of $N_i$. Since the sequence $N_1,\ldots,N_n$ is strongly Tor-independent, so is the sequence $N_1',\ldots, N_n'$. Moreover, strong Tor-independence implies that $\bigotimes_{i\in I} N_i'$ is a submodule of a free $R$-module, for each subset $i\in \{1,\ldots,n\}$. 

Use prime avoidance to find an $R$-regular element $x\in \fm_R-\fm_R^2$. Set $\overline{R}=R/xR$. Note that $\depth(\overline{R})=\depth(R)-1$ and $\operatorname{ecodepth}(\overline{R})=\operatorname{ecodepth}(R).$ The fact that each $\bigotimes_{i\in I} N_i'$ is a submodule of a free $R$-module implies that $x$ is also $\bigotimes_{i\in I} N_i'$-regular. It is straightforward to show that the sequence $\overline{R}\lotimes_R N_1', \ldots, \overline{R}\lotimes_R N_n'$ is strongly Tor-independent over $\overline{R}$.  By our induction hypothesis we have $$n\leq \operatorname{ecodepth}(\overline{R})=\operatorname{ecodepth}(R)$$ as desired. \qed


\begin{thebibliography}{1}

\bibitem{avramov:dgha}
L.~L. Avramov, H.-B.\ Foxby, and S.\ Halperin, \emph{Differential graded
  homological algebra}, in preparation.

\bibitem{avramov:htecdga}
L.~L. Avramov, S.~B. Iyengar, S.~Nasseh, and S.~Sather-Wagstaff, \emph{Homology
  over trivial extensions of commutative {DG} algebras}, Comm. Algebra
  \textbf{47} (2019), no.~6, 2341--2356, see also \texttt{arxiv:1508.00748}.
  \MR{3957101}

\bibitem{christensen:scatac}
L.~W. Christensen, \emph{Semi-dualizing complexes and their {A}uslander
  categories}, Trans. Amer. Math. Soc. \textbf{353} (2001), no.~5, 1839--1883.
  \MR{2002a:13017}

\bibitem{felix:rht}
Y.\ F{\'e}lix, S.\ Halperin, and J.-C.\ Thomas, \emph{Rational homotopy
  theory}, Graduate Texts in Mathematics, vol. 205, Springer-Verlag, New York,
  2001. \MR{1802847}

\bibitem{foxby:gmarm}
H.-B.\ Foxby, \emph{Gorenstein modules and related modules}, Math. Scand.
  \textbf{31} (1972), 267--284 (1973). \MR{48 \#6094}

\bibitem{gerko:sdc}
A.~A. Gerko, \emph{On the structure of the set of semidualizing complexes},
  Illinois J. Math. \textbf{48} (2004), no.~3, 965--976. \MR{2114263}

\bibitem{MR1132435}
A.~R. Kustin, \emph{Classification of the {T}or-algebras of codimension four
  almost complete intersections}, Trans. Amer. Math. Soc. \textbf{339} (1993),
  no.~1, 61--85. \MR{1132435}

\bibitem{nasseh:survey}
S.~Nasseh and S.~K. Sather-Wagstaff, \emph{Applications of differential graded
  algebra techniques in commutative algebra}, preprint (2020),
  \texttt{arXiv:2011.02065}.

\end{thebibliography}
\providecommand{\bysame}{\leavevmode\hbox to3em{\hrulefill}\thinspace}
\providecommand{\MR}{\relax\ifhmode\unskip\space\fi MR }
\providecommand{\MRhref}[2]{%
  \href{http://www.ams.org/mathscinet-getitem?mr=#1}{#2}
}
\providecommand{\href}[2]{#2}

\end{document}